\documentclass[a4paper,11pt,reqno,dvips]{amsart}
\usepackage{color}
\usepackage[hyphens]{url}
\let\savethebibliography=\thebibliography
\usepackage[square,numbers,comma,sort&compress]{natbib}
\let\thebibliography=\savethebibliography
\usepackage{enumerate}
\usepackage{colortbl}
\usepackage{booktabs}
\usepackage{ifpdf}
\usepackage{graphics}
\usepackage[all,line]{xy}

\def\ve#1{\mathchoice{\mbox{\boldmath$\displaystyle\bf#1$}}
{\mbox{\boldmath$\textstyle\bf#1$}}
{\mbox{\boldmath$\scriptstyle\bf#1$}}
{\mbox{\boldmath$\scriptscriptstyle\bf#1$}}}
\newcommand\vealpha{{\boldsymbol{\alpha}}}

\newcommand\Z{\mathbf Z}
\newcommand\N{\mathbf N}
\newcommand\R{\mathbf R}
\newcommand\Q{\mathbf Q}
\DeclareMathOperator{\ind}{ind}
\DeclareMathOperator{\sign}{sign}

\DeclareMathOperator{\cone}{cone}
\DeclareMathOperator{\interior}{int}
\DeclareMathOperator{\closure}{cl}
\DeclareMathOperator{\diagonal}{diag}
\let\boundary=\partial
\let\epsilon=\varepsilon
\usepackage{ifthen}
\makeatletter
\newcommand{\DeclareBracket}[3]{
  \newcommand{#1}[2][]{%
  \ifthenelse%
  {\equal{##1}{}}%
  {\left#2##2\right#3}%
  {\csname ##1l\endcsname#2##2\csname ##1r\endcsname#3}}}    
\makeatother
\DeclareBracket\abs||           % absolute value
\DeclareBracket\norm\|\|        % norm
\DeclareBracket\floor\lfloor\rfloor
\DeclareBracket\ceil\lceil\rceil
\DeclareBracket\set\{\}
\DeclareBracket\paren()
\DeclareBracket\inner\langle\rangle
\DeclareBracket\fractional\{\}
\newcommand{\cractional}[1]{\left\{\!\left\{#1\right\}\!\right\}}
\newcommand\C{\mathbf C}

\usepackage{amsthm}
\newtheorem{theorem}{Theorem}%
\makeatletter
\newtheorem{lemma}{Lemma}
\renewcommand*{\c@lemma}{\c@theorem}
\renewcommand*{\p@lemma}{\p@theorem}

\renewcommand*{\c@conjecture}{\c@theorem}
\renewcommand*{\p@conjecture}{\p@theorem}

\renewcommand*{\c@proposition}{\c@theorem}
\renewcommand*{\p@proposition}{\p@theorem}

\renewcommand*{\c@corollary}{\c@theorem}
\renewcommand*{\p@corollary}{\p@theorem}

\renewcommand*{\c@observation}{\c@theorem}
\renewcommand*{\p@observation}{\p@theorem}

\theoremstyle{definition}

\renewcommand*{\c@problem}{\c@theorem}
\renewcommand*{\p@problem}{\p@theorem}

\renewcommand*{\c@definition}{\c@theorem}
\renewcommand*{\p@definition}{\p@theorem}

\newtheorem{remark}{Remark}
\renewcommand*{\c@remark}{\c@theorem}
\renewcommand*{\p@remark}{\p@theorem}

\newtheorem{example}{Example}
\renewcommand*{\c@example}{\c@theorem}
\renewcommand*{\p@example}{\p@theorem}

\newtheorem{algorithm}{Algorithm}
\renewcommand*{\c@algorithm}{\c@theorem}
\renewcommand*{\p@algorithm}{\p@theorem}

\makeatother

\usepackage[breaklinks=true,colorlinks,citecolor=blue]{hyperref}

\title[Primal parametric Barvinok algorithm]
{Computing parametric rational generating functions
  with a primal Barvinok algorithm}

\author{Matthias~K\"oppe}
\address{Matthias~K\"oppe: Otto-von-Guericke-Universit\"at Magdeburg, Department of
  Mathematics, Institute for Mathematical Optimization (IMO),
  Univer\-si\-t\"ats\-platz~2, 
  39106 Magdeburg, Germany} 
\email{mkoeppe@imo.math.uni-magdeburg.de}
\thanks{The first author was supported by a 2006/2007 Feodor Lynen Research Fellowship from the
  Alexander von Humboldt Foundation.  He also acknowledges the
  hospitality of Jes\'us De Loera and the Department of Mathematics of the
  University of California, Davis, where a part of this work was completed.}

\author{Sven Verdoolaege}
\address{
Sven Verdoolaege: Leiden Institute of Advanced Computer Science (\mbox{LIACS}),
Universiteit Leiden, Niels Bohrweg 1, 2333 CA Leiden, The Netherlands
}
\email{sverdool@liacs.nl}

\date{$\relax$Revision: 1.99 $ - \ $Date: 2007/08/27 22:02:57 $ $}


\begin{document}

\begin{abstract}
  Computations with Barvinok's short rational generating functions are
  traditionally being performed in the dual space, to avoid the combinatorial complexity
  of inclusion--exclusion formulas for the intersecting proper faces of cones.
  We prove that, on the level of indicator functions of polyhedra, 
  there is no need for using inclusion--exclusion formulas to
  account for boundary effects:  All linear identities in the space of
  indicator functions can be purely expressed using half-open variants of the
  full-dimensional polyhedra in the identity.  This gives rise to a
  practically efficient, 
  parametric Barvinok algorithm in the primal space.
\end{abstract}

\maketitle

\section{Introduction}

We consider a family of polytopes $P_{\ve q} = \{\,\ve x\in\R^d: A\ve x \leq
\ve q\,\}$ parameterized by a right-hand side vector $\ve q\in Q\subseteq\R^m$,
where the set of right-hand sides is restricted to some polyhedron~$Q$. 
For this family of polytopes, we define the
\emph{parametric counting function} $c\colon Q \to \N$ by
\begin{equation}
  \label{eq:parametric-counting-function}
  c(\ve q) = \# \bigl( P_{\ve q} \cap \Z^d \bigr).
\end{equation}
Note that this includes vector partition functions
$
c(\ve \lambda) = \# \{\, \ve x \in \N^d : A' \ve x = \ve \lambda \,\}
$
as a special case.
It is well-known that the counting function \eqref{eq:parametric-counting-function}
is a piecewise quasipolynomial function.  
We are interested in computing an efficient algorithmic representation of the
function that allows to efficiently evaluate $c(\ve q)$ for any given~$\ve q$.
This paper builds on various techniques described in the literature, which we
review in the following.

\subsection{Barvinok's short rational generating functions}
The foundation of our method is an algorithmically efficient calculus of
\emph{rational generating
  functions} of the integer points in polyhedra developed by 
\citet{Barvinok94}; see also~\cite{BarviPom}.
Let $P = P_{\ve q}\subseteq\R^d$ be a rational polyhedron.  The
\emph{generating function} of~$P\cap\Z^d$ is defined as the formal Laurent series
\begin{displaymath}
  \tilde g_P(\ve z) = \sum_{\vealpha\in P\cap\Z^d} \ve z^{\vealpha}
\in \Z[[z_1,\dots,z_d, z_1^{-1},\dots,z_d^{-1}]],
\end{displaymath}
using 
the multi-exponent notation $\ve z^{\vealpha} = \prod_{i=1}^d z_i^{\alpha_i}$.
If $P$ is bounded, $\tilde g_P$ is a Laurent polynomial, which we
consider as a rational function~$g_P$.  If $P$ is not
bounded but is pointed (i.e., $P$ does not contain a straight line), there is
a non-empty open subset $U\subseteq\C^d$ such that the series converges
absolutely and uniformly on every compact subset of~$U$ to a rational
function~$g_P$.  If $P$ contains a straight line, we set $g_P = 0$.
The rational function $g_P\in\Q(z_1,\dots,z_d)$ defined in this way is called
the \emph{rational generating function} of~$P\cap\Z^d$.

By Brion's Theorem \cite{Brion88}, the rational generating function of a
polyhedron is the sum of the rational generating functions of its vertex
cones.  Thus the computation of a rational generating function can be reduced
to the case of \emph{polyhedral cones}.  
Moreover, the mapping $P \mapsto g_P$ is a %%(rational-function-valued)
\emph{valuation}: Let $[P]$ denote the \emph{indicator
  function} of~$P$, i.e., the function 
\begin{equation*}
  [P]\colon \R^d\to\R, \quad
  [P](\ve x) = 
  \begin{cases}
    1 & \text{if $\ve x\in P$} \\
    0 & \text{otherwise}.
  \end{cases}
\end{equation*}
The valuation property is that any (finite) linear identity 
\begin{math}
%%  \label{eq:indicator-identity}
  \sum_{i\in I} \epsilon_i [P_i] = 0
\end{math}
with $\epsilon_i\in\Q$
carries over to a linear identity 
\begin{math}
%%  \label{eq:ratgenfun-identity}
  \sum_{i\in I} \epsilon_i\, g_{P_i}(\ve z) = 0
\end{math}.  Hence, it is possible to use the inclusion--exclusion principle to break a
polyhedral cone into pieces and to add and subtract
the resulting generating
functions.  Indeed, by triangulating the vertex cones, one can reduce the
problem to the case of \emph{simplicial cones}. 

By elimination of variables it suffices to consider the case of
\emph{full-di\-men\-sion\-al} simplicial cones, i.e., cones $C\subseteq\R^d$
generated by $d$~linearly independent ray vectors $\ve b_1,\dots,\ve
b_d\in\Z^d$.  The \emph{index} of such a cone is defined as the index of the
point lattice generated by $\ve b_1,\dots,\ve b_d$ in the standard
lattice~$\Z^d$; we have $\ind C = \bigl|\det (\ve b_1, \dots, \ve b_d)\bigr|$.
Using Barvinok's \emph{signed decomposition technique}, it is possible to
write a cone as
\begin{displaymath}
  [C] = \sum_{i\in I_1} \epsilon_i [C_i] + \sum_{i\in I_2} \epsilon_i [C_i]
  \quad\text{with $\epsilon_i \in\{\pm1\}$},
\end{displaymath}
with at most $d$ full-dimensional simplicial cones $C_i$ of
lower index
in the sum over $i\in I_1$
and $\mathrm O(2^d)$  lower-dimensional simplicial cones $C_i$
in the sum over $i\in I_2$.  
The lower-dimensional cones arise due to the inclusion--exclusion principle
applied to the intersecting faces of the full-dimensional cones.
The signed decomposition is then recursively applied to the cones~$C_i$, until
one obtains unimodular (index~1) cones, for which the rational generating function can
be written down trivially.
Since the indexes of the full-dimensional cones descend quickly enough at each
level of the decomposition, one can prove the depth of the decomposition tree is doubly
logarithmic in the index of the input 
cone.  This gives rise to a polynomiality result \emph{in fixed dimension}: 
\begin{theorem}[\citet{Barvinok94}]
  Let the dimension~$d$ be fixed.  There exists a polynomial-time algorithm for computing 
  the rational generating function of a polyhedron $P\subseteq\R^d$ given by rational
  inequalities. 
\end{theorem}

Despite the polynomiality result, the algorithm was widely considered to be
practically inefficient because too many, $\mathrm{O}(2^d)$, lower-dimensional
cones had to be created at every level of the decomposition.  Later the
algorithm was improved by making use of Brion's ``polarization trick'',
see~\cite{Brion88} and~\citep[Remark 4.3]{BarviPom}:
The computations with rational generating
functions are invariant with respect to the contribution of non-pointed cones
(cones containing a non-trivial linear subspace).  The reason is that the
rational generating function of every non-pointed cone is zero.  By operating
in the dual space, i.e., by computing with the polars of all cones,
lower-dimensional cones can be safely discarded, because this is equivalent to
discarding non-pointed cones in the primal space.  Thus at each level of the
decomposition, only at most $d$ cones are created.  This \emph{dual variant}
of Barvinok's algorithm has efficient implementations in
LattE~\citep{latte-1.2,latte1,latte2} and the library \texttt{barvinok}~%
\citep{barvinok-manual-noversion}.

\subsection{Parametric polytopes and generating functions}
\label{s:parametric:polytope}

The vertices of a parametric polytope
$P_{\ve q} = \{\,\ve x\in\R^d: A\ve x \leq \ve q\,\}$, with
$\ve q\in Q\subseteq\R^m$ are affine functions of the parameters $\ve q$
and can be computed as follows.
A set $B$ of $d$ linearly independent rows of the inequality system $A\ve
x\leq\ve q$ is called a \emph{simplex basis}.  The associated \emph{basic
  solution} $\ve x(B)$ is the unique solution of the equation $A_B\ve x = \ve
q_B$.
Note that different simplex bases may give rise to the same basic solution.
A simplex basis (and the corresponding basic solution) is called \emph{(primal)
  feasible} if $A\ve x(B) \leq \ve q$ holds for some $\ve q \in Q$.
The vertices of $P_{\ve q}$ correspond to the feasible basic
solutions and they are said to be {\em active} on the subset of $Q$
for which the basic solutions are feasible.

A \emph{chamber} of the parameterized inequality system $A\ve x\leq\ve q$ is an
inclusion-maximal set of right-hand side vectors~$\ve q$ that have the same
set of primal feasible simplex bases.
The chamber complex of
$P_{\ve q}$ is the common refinement of the projections into $Q$
of the $n$-faces of the polyhedron
$\hat P = \{\, (\ve x, \ve q) \in \R^d \times Q : A\ve x \le \ve q\,\}$,
where $n$ is the dimension of the projection of $\hat P$ onto $Q$%
~\cite{Loechner97parameterized,verdoolaege-woods-2005}.
Alternatively, the problem may be translated into
a vector partition problem, for which the chambers can be
computed either directly~\cite{Baldoni2001counting}
or as the regular triangulations of its
Gale transform~\cite{Gelfand1994,Pfeifle2003}.
However, these alternative computations, discussed in more
detail in~\cite{Eisenschmidt2007integrally,barvinok-manual-noversion},
may lead to many chambers that do not meet $Q$ and
that hence have to be discarded.

Within each (open) chamber of the chamber complex,
the combinatorial type of $P_{\ve q}$
remains the same and Barvinok's algorithm can be applied to the vertices
active on the chamber~\cite[Theorem~5.3]{BarviPom}.
As we will explain in more detail in \autoref{s:gf},
the result is a parametric rational generating function where the
parameters only appear in the numerator.
In practice, it is sufficient to apply Barvinok's algorithm in the closures
of the chambers of maximal dimension~\cite[Section~4.2]{Brion1997residue}.
On intersections of these closures
one obtains possibly different representations of the same
parametric rational generating function.

\begin{example}
As a trivial example, consider the one-dimensional
parametric polytope
$
P_q = \{\, x \in \R^1: x \geq 0, \;
2 x \leq  q + 6 ,\;
x \leq q \,\}
$.
Its vertices are $0$, $q/2 + 3$ and $q$, active on $\{\, q \ge 0 \,\}$,
$\{\, q \ge 6 \,\}$ and $\{\, q \le 6 \,\}$, respectively.
The full-dimensional (open) chambers are $\{\, 0 < q < 6 \,\}$
and $\{\, q > 6 \,\}$ and the resulting parametric counting function is
$$
c(q) =
\begin{cases}
q+1 & \hbox{if $0 \le q \le 6$}
\\
\floor{\frac{q}2}  + 4  & \hbox{if $6 \le q$}
.
\end{cases}
$$
\end{example}

As in the non-parametric case, $P_{\ve q}$ can be assumed to be full-dimensional
for all parameter values in the chambers of maximal dimension.
Note that a reduction to the full-dimensional case may involve a
reduction of the parameters to the standard lattice%
~\cite{verdoolaege-et-al:counting-parametric,Meister2004PhD}.
This parametric version of the dual variant of Barvinok's algorithm
has also been implemented in \texttt{barvinok}~\cite{barvinok-manual-noversion}
and is explained in more detail in~\cite{verdoolaege-et-al:counting-parametric,%
verdoolaege-woods-2005,verdoolaege-wood-bruynooghe-cools-2005}.

\subsection{Irrational decompositions and primal algorithms}
Recently, \citet{beck-sottile:irrational} introduced \emph{irrational triangulations} of
polyhedral cones as a technique for obtaining simplified proofs for theorems
on generating functions.  Let $\ve v + C \subseteq\R^d$ be a
full-dimensional affine polyhedral cone; it can be triangulated into simplicial
full-dimensional cones $\ve v + C_i$.  Then there exists a vector $\ve{\tilde
  v}\in\R^d$ such that 
\begin{equation}
  (\ve{\tilde v} + C) \cap \Z^d = (\ve v + C) \cap \Z^d
\end{equation}
and
\begin{equation}
  \label{eq:integrally-empty-boundary}
  \boundary (\ve{\tilde v} + C_i) \cap \Z^d = \emptyset, 
\end{equation}
that is, the affine cones $\ve{\tilde v} + C_i$ do not have any \emph{integer
  points} in common. Thus, without using the inclusion--exclusion principle,
one obtains an identity on the level of generating functions,
\begin{equation}
  g_{\ve v + C}(\ve z) = g_{\ve{\tilde v} + C}(\ve z)  = \sum_i g_{\ve{\tilde
      v} + C_i}(\ve z). 
\end{equation}

\citet{koeppe:irrational-barvinok} considered both irrational triangulations
and \emph{irrational signed decompositions}.  He constructed a \emph{uniform}
irrational shifting vector $\ve{\tilde v}$ which ensures that
\eqref{eq:integrally-empty-boundary}~holds for all cones $\ve{\tilde v} + C_i$
that are created during the course of the recursive Barvinok decomposition
method.  The implementation of this method in a version of
LattE~\cite{latte-macchiato} was the first practically efficient variant of 
Barvinok's algorithm that works in the primal space. 

The benefits of a decomposition in the primal space are twofold.  First, it
allows to effectively use the method of \emph{stopped decomposition}
\citep{koeppe:irrational-barvinok}, where the recursive decomposition of the
cones is stopped before unimodular cones are obtained.  For certain classes of
polyhedra, this technique reduces the running time by several orders of
magnitude.

Second, for some classes of polyhedra such as the cross-polytopes, it is
prohibitively expensive to compute triangulations of the vertex cones in the
dual space.  An \emph{all-primal algorithm}
\citep{koeppe:irrational-barvinok} that computes both triangulations and
signed decompositions in the primal space is therefore able to handle problem
instances that cannot be solved with a dual algorithm in reasonable time.

\subsection{The contribution of this paper.}
The irrationalization technique of
\cite{beck-sottile:irrational,koeppe:irrational-barvinok} can be viewed as a method
of translating an \emph{inexact identity} (i.e., an identity modulo the
contribution of lower-dimensional cones) of indicator functions of
full-dimensional cones, 
\begin{equation}
  \label{eq:source-identity-with-mod}
  \sum_{i\in I} \epsilon_i [\ve v_i + C_i] \equiv 0
  \pmod{\text{lower-dimensional cones}}
\end{equation}
to an exact identity of rational generating functions,
\begin{equation}
  \sum_{i\in I} \epsilon_i\, g_{\ve{\tilde v}_i + C_i}(\ve z) = 0.
\end{equation}
We remark that this identity is not valid on the level of indicator functions.
In contrast, in \autoref{s:identities} we provide a general constructive method of translating an inexact
identity~\eqref{eq:source-identity-with-mod} of 
indicator functions of full-dimensional cones to an exact identity of
indicator functions of full-dimensional \emph{half-open} cones,
\begin{equation}
  \sum_{i\in I} \epsilon_i [\ve v_i + \tilde C_i]  = 0,
\end{equation}
without increasing the number of summands in the identity. 

This general result gives rise to methods of exact
polyhedral subdivision  of polyhedral cones (\autoref{s:subdivision}) and exact signed
decomposition of half-open simplicial cones (\autoref{s:decomposition}).

Since the rational generating function of half-open simplicial cones of low
index can be written down easily (\autoref{s:gf}), we obtain new primal
variants of Barvinok's algorithm.  The new variants have simpler
implementations than the 
primal irrational variant \cite[Algorithm~5.1]{koeppe:irrational-barvinok} 
and the all-primal irrational variant
\cite[Algorithm~6.4]{koeppe:irrational-barvinok} 
because computations with large rational numbers can be replaced by simple,
combinatorial rules.  
%% benefit for nonparametric computation:  Computation of stability cube by
%% LP in ``all-primal'' algorithm eliminated. (We just need a suitable $\rho$!)

The new variants based on exact decomposition in the primal space
are particularly useful for \emph{parametric} problems.  The reason is that
the method of constructing the half-open polyhedral cones only depends on the 
facet normals and is independent from the location of the parametric vertex.
In contrast, the irrationalization technique needs to shift the parametric
vertex by a vector~$\ve s$ which needs to depend on the parameters.  This is
of particular importance for the case of the irrational all-primal algorithm,
where the irrational shifting vector~$\ve s$ needs to be constructed by solving a
parametric linear program. 

Moreover, the technique of exact decomposition can also be applied to the
parameter space~$Q$, obtaining a partition into half-open chambers~$\tilde
Q_i$. This gives rise to useful new representations of the parametric
generating function $g_{P_{\ve q}}(\ve z)$ (\autoref{s:representations}) and
the counting function~$c(\ve q)$ (\autoref{s:specialization}).
We also introduce algorithmic representations of $g_{P_{\ve q}}(\ve z)$ and
$c(\ve q)$ that make use of half-open activity domains of the parametric
vertices.  Its benefit is that it is of polynomial size and has polynomial
evaluation time even when the dimension~$m$ of the parameter space varies.
\smallbreak

Taking all together, we obtain the first practically efficient parametric
Barvinok algorithm in the primal space.

\section{Exact triangulations and signed decompositions into half-open polyhedra}

\subsection{Identities in the algebra of indicator functions, or:
  Inclusion--exclusion is not hard for boundary effects}
\label{s:identities}

We first show that identities of indicator functions of full-dimensional
polyhedra modulo lower-dimensional polyhedra can be translated to \emph{exact}
identities of indicator functions of full-dimensional half-open polyhedra. 

\begin{theorem}
  \label{th:exactify-identities}
  Let 
  \begin{equation}
    \label{eq:full-source-identity}
    \sum_{i\in I_1} \epsilon_i [P_i] + \sum_{i\in I_2} \epsilon_i [P_i] = 0
  \end{equation}
  be a (finite) linear identity of indicator functions of closed
  polyhedra~$P_i\subseteq\R^d$, where the
  polyhedra~$P_i$ are full-dimensional for $i\in I_1$ and 
  lower-dimensional for $i\in I_2$, and where $\epsilon_i\in\Q$.  Let each closed polyhedron be given as 
  \begin{align}
    P_i &= \bigl\{\, \ve x : \langle \ve b^*_{i,j}, \ve x\rangle \leq \beta_{i,j} \text{
      for $j\in J_i$}\,\bigr\}.
  \end{align}
  Let $\ve y\in\R^d$ be a vector such that $\langle \ve b^*_{i,j}, \ve
  y\rangle \neq 0$ for all $i\in I_1\cup I_2$, $j\in J_i$.
  For $i\in I_1$, we define the half-open polyhedron
  \begin{equation}
    \label{eq:half-open-by-y}
    \begin{aligned}
      \tilde P_i = \Bigl\{\, \ve x\in\R^d : {}& \langle \ve b^*_{i,j}, \ve x\rangle \leq \beta_{i,j}
      \text{ for $j\in J_i$ with $\langle \ve b^*_{i,j}, \ve y \rangle < 0$,} \\
      & \langle \ve b^*_{i,j}, \ve x\rangle < \beta_{i,j}
      \text{ for $j\in J_i$ with $\langle \ve b^*_{i,j}, \ve y \rangle > 0$} \,\Bigr\}.
    \end{aligned}
  \end{equation}
  Then 
  \begin{equation}
    \label{eq:target-identity}
    \sum_{i\in I_1} \epsilon_i [\tilde P_i] = 0.
  \end{equation}
\end{theorem}
\begin{proof}
  We will show that \eqref{eq:target-identity} holds for an arbitrary~$\ve{\bar
    x}\in\R^d$. To this end, fix an arbitrary $\ve{\bar x}\in\R^d$.  We define
  \begin{displaymath}
    \ve x_\lambda = \ve{\bar x} + \lambda\ve y\quad\text{for $\lambda\in[0,+\infty)$}.
  \end{displaymath}
  Consider the function
  \begin{displaymath}
    f\colon [0,+\infty)\ni\lambda \mapsto \biggl(\sum_{i\in I_1} \epsilon_i
    [\tilde P_i]\biggr)(\ve x_\lambda).
  \end{displaymath}
  We need to show that $f(0) = 0$.
  To this end, we first show that $f$ is constant in a neighborhood of~$0$.

  First, let $i\in I_1$ such that $\ve{\bar x} \in \tilde P_i$.  
  For $j\in J_i$ with $\langle \ve b^*_{i,j}, \ve y\rangle < 0$, 
  we have $\langle \ve b^*_{i,j}, \ve{\bar x} \rangle \leq \beta_{i,j}$, thus $\langle \ve
  b^*_{i,j}, \ve{x}_\lambda \rangle \leq \beta_{i,j}$. 
  For $j\in J_i$ with $\langle \ve b^*_{i,j}, \ve y\rangle > 0$, we 
  have $\langle \ve b^*_{i,j}, \ve{\bar x} \rangle < \beta_{i,j}$, thus $\langle \ve
  b^*_{i,j}, \ve{x}_\lambda \rangle < \beta_{i,j}$ for $\lambda>0$ small enough.
  Hence, $\ve{x}_\lambda \in \tilde P_i$ for $\lambda>0$ small enough.

  Second, let $i\in I_1$ such that $\ve{\bar x} \notin \tilde P_i$.  
  Then either there exists a $j\in J_i$ with
  $\langle \ve b^*_{i,j}, \ve y\rangle < 0$ and $\langle \ve b^*_{i,j},
  \ve{\bar x}\rangle > \beta_{i,j}$.  Then $\langle \ve
  b^*_{i,j}, \ve{x}_\lambda \rangle > \beta_{i,j}$ for $\lambda>0$ small enough.
  Otherwise, there exists a $j\in J_i$ with
  $\langle \ve b^*_{i,j}, \ve y\rangle > 0$ and $\langle \ve b^*_{i,j},
  \ve{\bar x}\rangle \geq \beta_{i,j}$.  Then $\langle \ve
  b^*_{i,j}, \ve{x}_\lambda \rangle \geq \beta_{i,j}$.
  Hence, in either case, $\ve{x}_\lambda \notin \tilde P_i$ for $\lambda>0$
  small enough.\smallbreak

  Next we show that $f$ vanishes on some interval $(0, \lambda_0)$. 
  We consider the function 
  \begin{displaymath}
    g\colon [0,+\infty)\ni\lambda \mapsto \biggl(\sum_{i\in I_1} \epsilon_i
    [P_i] + \sum_{i\in I_2} \epsilon_i [P_i]\biggr)(\ve x_\lambda)
  \end{displaymath}
  which is constantly zero by~\eqref{eq:full-source-identity}.  
  Since $[P_i](\ve x_\lambda)$ for $i\in I_2$ vanishes on all
  but finitely many $\lambda$, we have 
  \begin{displaymath}
    g(\lambda) = \biggl(\sum_{i\in I_1} \epsilon_i
    [P_i]\biggr) (\ve{x}_\lambda)
  \end{displaymath}
  for $\lambda$ from some interval $(0,\lambda_1)$.  Also, $[P_i](\ve x_\lambda)
  = [\tilde P_i](\ve x_\lambda)$ for some interval $(0, \lambda_2)$.  Hence
  $f(\lambda) = g(\lambda) = 0$ for some interval $(0, \lambda_0)$.\smallbreak
  
  Hence, since $f$ is constant in a neighborhood of~$0$, it is also zero at
  $\lambda=0$.
  Thus the identity~\eqref{eq:target-identity} holds for~$\ve{\bar x}$.
\end{proof}

\begin{remark}
  \autoref{th:exactify-identities} can be easily generalized to a situation
  where the weights $\epsilon_i$ are not constants but
  continuous real-valued functions.  In the proof, rather than showing that
  $f$ is constant in a neighborhood of~$0$, one shows that $f$ is continuous at~$0$.
\end{remark}

\subsection{The exact polyhedral subdivision of a closed polyhedral cone}
\label{s:subdivision}

For obtaining an exact polyhedral subdivision of a full-dimensional closed polyhedral
cone $C = \cone\{\ve b_1,\dots,\ve b_n\}$, 
\begin{displaymath}
  [C] = \sum_{i\in I_1} [\tilde C_i],
\end{displaymath}
we apply the above theorem using an arbitrary vector $\ve y\in \interior C$
that avoids all facets of the cones~$C_i$,
for instance 
\begin{displaymath}
  \ve y = \sum_{i=1}^n (1+\gamma^i) \ve b_i
\end{displaymath}
for a suitable~$\gamma>0$.

\subsection{The exact signed decomposition of half-open simplicial cones}
\label{s:decomposition}

Let $\tilde C\subseteq\R^d$ be a half-open simplicial full-dimensional cone
with the double description 
  \begin{align}
    \tilde C &= \Bigl\{\, \ve x\in\R^d : {}
    \langle \ve b^*_{j}, \ve x\rangle \leq 0
    \text{ for $j\in J_\leq$ and}\  
    %%\\ &
    \langle \ve b^*_{j}, \ve x\rangle < 0
    \text{ for $j\in J_<$} \,\Bigr\} \\
\label{eq:explicit}
    \tilde C&= \Bigl\{\, \textstyle\sum_{j=1}^d \lambda_j \ve b_j :{} 
    \lambda_j \geq 0 \text{ for $j\in J_\leq$ and} \ 
  %% \\ &
  \lambda_j > 0 \text{ for $j\in J_<$} 
  \,\Bigr\}
  \end{align}
where $J_< \cup
J_{\leq} = \{1,\dots,d\}$, with the \emph{biorthogonality property} for the
outer normal vectors $\ve b^*_j$ and the ray vectors~$\ve b_i$,
\begin{equation}
  \label{eq:biorthogonality}
  \langle \ve b^*_j, \ve b_i \rangle = -\delta_{i,j}
  =
  \begin{cases}
    -1 & \text{if $i=j$}, \\
    0  & \text{otherwise}.
  \end{cases}
\end{equation}
In the following we introduce a generalization of Barvinok's \emph{signed decomposition}
\citep{Barvinok94} to half-open simplicial cones~$C_i$, which will give an exact
identity of half-open cones.  To this end, we first compute the usual signed
decomposition of the closed cone $C = \closure {\tilde C}$, 
\begin{equation}\label{eq:signed-decomposition-inexact-identity}
  [C] \equiv \sum_{i}\epsilon_i [C_i] \pmod{\text{lower-dimensional
      cones}} 
\end{equation}
using an extra ray~$\ve w$, which has the representation
\begin{equation}
  \ve w = \sum_{i=1}^d \alpha_i \ve b_i \quad\text{where $\alpha_i =
    -\langle \ve b^*_i, \ve w \rangle$.}
\end{equation}
Each of the cones $C_i$ is spanned by
$d$~vectors from the set $\{\ve b_1, \dots, \ve b_d, \ve w\}$.  The signs
$\epsilon_i\in\{\pm1\}$ are determined according to the location of~$\ve w$,
see \citep{Barvinok94}.

An exact identity
\begin{displaymath}
  [\tilde C] = \sum_{i}\epsilon_i [\tilde C_i] \quad\text{with $\epsilon\in\{\pm1\}$},
\end{displaymath}
can now be obtained from~\eqref{eq:signed-decomposition-inexact-identity}
as follows.  We define cones~$\tilde C_i$ that are half-open counterparts
of~$C_i$.  We only need to determine which of the defining inequalities of the
cones $\tilde C_i$ should be strict.  
To this end, we first show how to construct a vector~$\ve y$ that characterizes which defining
inequalities of~$\tilde C$ are strict by the means of~\eqref{eq:half-open-by-y}.
\begin{lemma}
  \label{lemma:reverse-engineering-y}
  Let
  \begin{equation}
    \label{eq:sigma}
    \sigma_i =
    \begin{cases}
      1 & \text{for $i\in J_\leq$,} \\
      -1 & \text{for $i\in J_<$,}
    \end{cases}
  \end{equation} 
  and let $\ve y \in R = \interior\cone\{\, \sigma_1 \ve b_i, \dots,
  \sigma_d \ve b_d\}$ be arbitrary.
  Then
  \begin{align*}
    J_{\leq} &= \bigl\{\, j\in\{1,\dots,d\} : 
    \langle \ve b^*_{j}, \ve y \rangle < 0 \,\bigr\}, \\
    %%\intertext{and}
    J_{<} &= \bigl\{\, j\in\{1,\dots,d\} : 
    \langle \ve b^*_{j}, \ve y \rangle > 0 \,\bigr\}.
  \end{align*}
\end{lemma}
We remark that the construction of such a vector~$\ve y$ is not possible for
a half-open non-simplicial cone in general.
\begin{proof}[Proof of \autoref{lemma:reverse-engineering-y}]
  Such a $\ve y$ has the representation 
  \begin{displaymath}
    \ve y = \sum_{i\in J_\leq} \lambda_i
    \ve b_i - \sum_{i\in J_<} \lambda_i \ve b_i
    \quad\text{ with $\lambda_i>0$}.
  \end{displaymath}
  Thus
  \begin{displaymath}
    \langle \ve b^*_{j}, \ve y \rangle
    =
    \begin{cases}
      -\lambda_j & \text{for $j\in J_{\leq}$,} \\
      +\lambda_j & \text{for $j\in J_{<}$,}
    \end{cases}
  \end{displaymath}
  which proves the claim.
\end{proof}

Now let $\ve y \in R$ be an arbitrary vector that is not orthogonal to any of
the facets of the cones~$\tilde C_i$.   Then such a vector~$\ve y$ can determine
which of the defining inequalities of the cones~$\tilde C_i$ are strict.  \smallbreak

In the following, we give a specific construction of such a vector~$\ve y$.
To this end,
let $\ve b_m$ be the unique ray of~$\tilde C$ that is not a ray of~$\tilde
C_i$.  Then we denote by $\ve{\tilde b}^*_{0,m}$ the outer normal vector of
the unique facet of~$\tilde C_i$ not incident to~$\ve w$.
Now consider any facet~$F$ of a
cone~$\tilde C_i$ that is incident to~$\ve w$.  Since $\tilde
C_i$~is simplicial, there is exactly one ray of~$\tilde C_i$, say~$\ve b_l$,
not incident to~$F$.  The outer normal vector of the facet is
therefore characterized up to scale by the indices $l$ and $m$; thus we denote it by
$\ve{\tilde b}^*_{l,m}$.
See \autoref{fig:decomposition} for an example of this naming convention.

Let $\ve b_0 = \ve w$.  Then, for every outer normal vector~$\ve{\tilde
  b}^*_{l,m}$ and every ray~$\ve b_i$, $i = 0,\dots, d$, we have 
\begin{equation}
  \label{eq:signed-decomp:extra-facet-scalar-signs}
  \beta_{i;l,m} := -\langle \ve{\tilde b}^*_{l,m}, \ve b_i \rangle
  \begin{cases}
    > 0 & \text{for $i = l$,} \\
    = 0 & \text{for $i \neq l, m$,} \\
    \in \R & \text{for $i = m$.} % It CAN also lie on the facet. --mkoeppe
  \end{cases}
\end{equation}
Now the outer normal vector has the representation
\begin{equation*}
  \ve{\tilde b}^*_{l,m} = \sum_{i=1}^d \beta_{i;l,m} \ve b^*_i.
\end{equation*}
%% The equation $\langle \ve{\tilde b}^*_{l,m}, \ve w \rangle = 0$ and 
The conditions of~\eqref{eq:signed-decomp:extra-facet-scalar-signs} determine
the outer normal vector~$\ve{\tilde b}^*_{l,m}$ up to scale.  
For the normals~$\ve{\tilde b}^*_{0,m}$, we can choose
\begin{equation}
  %%\beta_{m;0,m} = \alpha_m.
  \ve{\tilde b}^*_{0,m} = \alpha_m \ve b^*_m.
\end{equation}
%% First, $\beta_{i;l,m} = 0$
%% for $i\neq l,m$ due to~\eqref{eq:signed-decomp:extra-facet-scalar-signs}.   
For the other facets~$\ve{\tilde b}^*_{l,m}$, 
we can choose
\begin{equation}
  %%\beta_{l;l,m} = \abs{\alpha_m} \quad\text{and}\quad \beta_{m;l,m} = -\sign
  %%\alpha_m \cdot \alpha_l. 
   \ve{\tilde b}^*_{l,m} 
   = \abs{\alpha_m} \ve b^*_l  -\sign\alpha_m \cdot \alpha_l \ve b^*_m.
\end{equation}
%% Indeed we have
%% \begin{displaymath}
%%      \langle \ve{\tilde b}^*_{l,m}, \ve w \rangle
%%      = - \alpha_l \beta_{l;l,m} - \alpha_m \beta_{m;l,m}
%%      = - \alpha_l \abs{\alpha_m} + \abs{\alpha_m} \alpha_l = 0,
%%   \end{displaymath}
%%   and also $\langle \ve{\tilde b}^*_{l,m}, \ve b_l \rangle < 0$,
%%   so~\eqref{eq:signed-decomp:extra-facet-scalar-signs} is satisfied. 

Now consider
\begin{equation}
  \label{eq:signed-decomp:lexico-y-revised}
  \ve y = \sum_{i=1}^d \sigma_i (\abs{\alpha_i} + \gamma^i) \ve b_i,
\end{equation}
which lies in the cone~$R$ for every $\gamma>0$.
We obtain
\begin{equation}
    \langle \ve{\tilde b}^*_{0,m}, \ve y\rangle 
= - \sigma_m \alpha_m (\abs{\alpha_m} + \gamma^m)
    \label{eq:signed-decomp:extra-facet-scalar-with-betas:0}
\end{equation}
and
\begin{align}
    \langle \ve{\tilde b}^*_{l,m}, \ve y\rangle 
%%     &= \beta_{l;l,m} \langle \ve b^*_l,
%%     \ve y \rangle + \beta_{m;l,m} \langle \ve b^*_m, \ve y \rangle \notag\\
    &= \abs{\alpha_m} \langle \ve b^*_l, \ve y \rangle -\sign \alpha_m \cdot
    \alpha_l \langle \ve b^*_m, \ve y \rangle \notag\\
    &= - \abs{\alpha_m} \sigma_l (\abs{\alpha_l} + \gamma^l) 
    + \sign \alpha_m \cdot \alpha_l \sigma_m (\abs{\alpha_m} + \gamma^m)\notag\\
    &= ( \sign(\alpha_l\alpha_m)\sigma_m - \sigma_l) \abs{\alpha_l}
    \abs{\alpha_m} \notag\\
    &\qquad - \sigma_l \abs{\alpha_m} \gamma^l
    + \sign(\alpha_l \alpha_m) \sigma_m \abs{\alpha_l} \gamma^m,
    \label{eq:signed-decomp:extra-facet-scalar-with-betas}
  \end{align} 
for $l \ne 0$.
  The right-hand side of~\eqref{eq:signed-decomp:extra-facet-scalar-with-betas},
as a polynomial in~$\gamma$, only has finitely many roots.  Thus there are
only finitely 
many values of~$\gamma$ for which a scalar product $\langle \ve{\tilde b}^*_{l,m}, \ve y
\rangle$ can vanish for any of the finitely many facet normals $\ve{\tilde b}^*_{l,m}$.  
Let $\gamma>0$ be an arbitrary number for which none of the scalar products
vanishes.  Then the vector $\ve y$ defined by~\eqref{eq:signed-decomp:lexico-y-revised}
determines which of the defining inequalities of the cones $\tilde C_i$ should
be strict. 

\begin{remark}
  It is possible to construct an a-priori vector $\ve y$ that is suitable
  to determine which definining inequalities are strict for all the cones that
  arise in the hierarchy of triangulations and signed decompositions of a cone
  $C = \cone\{\ve b_1, \dots, \ve b_n\}$ in
  Barvinok's algorithm.  The construction uses the methods
  from~\cite{koeppe:irrational-barvinok}.  
%%   Let $\ve{\hat y}\in\R^d$ and $\rho>0$ such that the open cube $\ve{\hat y}
%%   + B_\infty(\rho)$ is contained in $C$.
  Let $0<r\in\Z$ and $\ve{\hat y}\in\frac1r\Z^d$ and such that the open cube $\ve{\hat y}
  + B_\infty(\frac 1r)$ is contained in $C$. (For instance, choose $\ve{\hat
    y} = \sum_{i=1}^n \ve b_i$ and choose $r$ large enough.) 
  Let $D$ be an upper bound on the determinant of any simplicial cone that
  can arise in a triangulation of~$C$, for instance
  \begin{equation}
    \label{eq:triangulated-index-bound}
    D = \left( \max\nolimits_{i=1}^n \norm{\ve b_i}^2 \right)^{n/2}
  \end{equation}
  by Lemma~16 of~\cite{koeppe:irrational-barvinok}. 
  Let $C = \max_{i=1}^n \norm{\ve b_i}_\infty$.  Using the data from
  Theorem~11 of~\cite{koeppe:irrational-barvinok}, 
  \begin{equation*}
    k = \floor{1 + \frac{ \log_2 \log_2 D }{\log_2 \frac{d}{d-1} } },
    \quad M = 2 (d-1)! ( d^{k} C )^{d-1}, 
  \end{equation*}
  we define
  \begin{displaymath}
    \ve s = \frac1r \cdot\left(\frac1{(2M)^1},\frac1{(2M)^2}, \dots, \frac1{(2M)^d} \right).
  \end{displaymath}
  Finally let $\ve{y} = \ve{\hat y} + \ve s$.  Then $\langle \ve b^*, \ve y
  \rangle \neq 0$ for any of the facet normals $\ve b^*$ that can arise in the
  hierarchy of triangulations and signed decompositions of the cone~$C$.
\end{remark}

\begin{remark}
  For performing the exact signed decomposition in a software implementation,
  it is not actually necessary to construct the vector~$\ve y$ and to evaluate
  scalar products.  In the following, we show that we can devise simple,
  ``combinatorial'' rules to decide  which defining
  inequalities should be strict.    To this end, let $\gamma > 0$ in~\eqref{eq:signed-decomp:lexico-y-revised}
  be small
  enough that none of the signs
  \begin{displaymath}
    \sigma_{l,m} = -\sign \langle \ve{\tilde b}^*_{l,m}, \ve y\rangle
  \end{displaymath}
  given by~\eqref{eq:signed-decomp:extra-facet-scalar-with-betas}
  change if $\gamma$ is decreased even more.  We can now determine 
  $\sigma_{l,m}$ for all possible cases.\smallbreak

  \noindent\textit{Case 0: $\alpha_m = 0$.} The cone would be
  lower-dimensional in this case, since $\ve w$ lies in the space spanned by
  the ray vectors except $\ve b_m$, and is hence discarded.

  \noindent\textit{Case 1: $l = 0$.}
    From \eqref{eq:signed-decomp:extra-facet-scalar-with-betas:0}, we have
    $$
	\sigma_{0,m} = \sign(\alpha_m) \sigma_m
    .
    $$
  
  \noindent\textit{Case 2: $l \ne 0$, $\alpha_l = 0$, $\alpha_m \neq 0$.}  Here we have $\langle \ve{\tilde b}^*_{l,m},
  \ve y\rangle = - \sigma_l \abs{\alpha_m} \gamma^l$, thus
  \begin{displaymath}
    \sigma_{l,m} = \sigma_l.
  \end{displaymath}
  
  \noindent\textit{Case 3: $l \ne 0$, $\alpha_l \alpha_m > 0$.} 
  In this case \eqref{eq:signed-decomp:extra-facet-scalar-with-betas}
  simplifies to
  \begin{equation}
    \label{eq:case2}
    \langle \ve{\tilde b}^*_{l,m}, \ve y\rangle 
    = ( \sigma_m - \sigma_l) \abs{\alpha_l} \abs{\alpha_m}
    - \sigma_l \abs{\alpha_m} \gamma^l
    + \sigma_m \abs{\alpha_l} \gamma^m.
  \end{equation}
    
  \noindent\textit{Case 3\,a: $\sigma_l = \sigma_m$.}  Here the first term
  of~\eqref{eq:case2} cancels, so
  \begin{displaymath}
    \sigma_{l,m} = -\sign \langle \ve{\tilde b}^*_{l,m}, \ve y\rangle =
    \begin{cases}
      1 & \text{if $l < m$,} \\
      -1  & \text{if $l > m$.}
    \end{cases}
  \end{displaymath}
  \noindent\textit{Case 3\,b: $\sigma_l \neq \sigma_m$.}
  Here the first term of~\eqref{eq:case2} dominates, so 
  \begin{displaymath}
    \sigma_{l,m} = -\sign \langle \ve{\tilde b}^*_{l,m}, \ve y\rangle = \sigma_l.
  \end{displaymath}

  \noindent\textit{Case 4: $l \ne 0$, $\alpha_l \alpha_m < 0$.}  
  In this case \eqref{eq:signed-decomp:extra-facet-scalar-with-betas}
  simplifies to
  \begin{equation}
    \label{eq:case3}
    \langle \ve{\tilde b}^*_{l,m}, \ve y\rangle 
    = - ( \sigma_m + \sigma_l) \abs{\alpha_l} \abs{\alpha_m}
    - \sigma_l \abs{\alpha_m} \gamma^l
    - \sigma_m \abs{\alpha_l} \gamma^m.
  \end{equation}
  
  \noindent\textit{Case 4\,a: $\sigma_l = \sigma_m$.}
  Here the first term
  of~\eqref{eq:case3} dominates, so
  \begin{displaymath}
    \sigma_{l,m} = \sigma_l = \sigma_m.
  \end{displaymath}

  \noindent\textit{Case 4\,b: $\sigma_l \neq \sigma_m$.}
  Here the first term
  of~\eqref{eq:case3} cancels, so
  \begin{displaymath}
    \sigma_{l,m} = -\sign \langle \ve{\tilde b}^*_{l,m}, \ve y\rangle =
    \begin{cases}
      \sigma_l & \text{if $l< m$,} \\
      \sigma_m & \text{if $l> m$.}
    \end{cases}
  \end{displaymath}

\begin{example}
Consider (the vertex figure of) the three-dimensional
cone on the left of~\autoref{fig:decomposition}.
The open and closed facets can be described as
$$
\sigma_1 = -1 \qquad \sigma_2 = 1 \qquad \sigma_3 = -1
,
$$
while the extra ray~$\ve w = \sum_{i=1}^d \alpha_i \ve b_i$
is such that
$$
\alpha_1 < 0 \qquad \alpha_2 > 0 \qquad \alpha_3 > 0
.
$$
For the facets of the cones in the decomposition we have
\begin{align*}
\sigma_{0,3} & \stackrel{1}{=} \sigma_3 = -1 &
\qquad
\sigma_{0,2} & \stackrel{1}{=} \sigma_2 = 1 &
\qquad
\sigma_{0,1} & \stackrel{1}{=} -\sigma_1 = 1
\\
\sigma_{1,3} & \stackrel{4a}{=} \sigma_1 = -1 &
\qquad
\sigma_{1,2} & \stackrel{4b}{=} \sigma_1 = -1 &
\qquad
\sigma_{2,1} & \stackrel{4b}{=} \sigma_1 = -1
\\
\sigma_{2,3} & \stackrel{3b}{=} \sigma_2 = 1 &
\qquad
\sigma_{3,2} & \stackrel{3b}{=} \sigma_3 = -1 &
\qquad
\sigma_{3,1} & \stackrel{4a}{=} \sigma_1 = -1
.
\end{align*}
The result is shown on the right of~\autoref{fig:decomposition}.

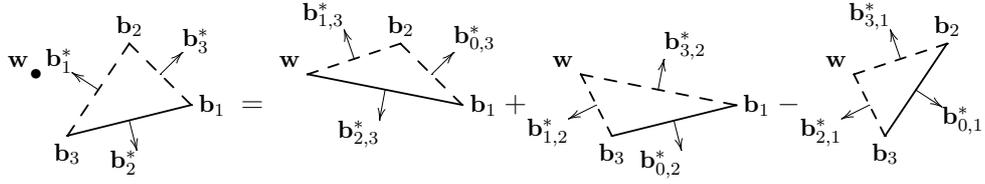
\begin{figure}
\newdimen\intercol
\intercol=0.48cm
$$
\begin{xy}
<\intercol,0pt>:<0pt,\intercol>::
\POS(0,0)*[*0.85]\xybox{
\POS(-2,-1)="b3"*+!U{\ve b_3}
\POS(2,0)="b1"*+!L{\ve b_1}
\POS(0,2)="b2"*+!D{\ve b_2}
\POS(-3,1)="w"*+!DR{\ve w}
\POS(-3,1)="w"*{\bullet}
\POS"b3"\ar@[|(2)]@{-}"b1"
\POS?(0.5)="a",?(0.5)/-\intercol/="b","a"
*\xybox{"b"-"a":(0,0)\ar(0,1)\POS(0,1)*!RU{\ve b_2^*}}%
\POS"b1"\ar@[|(2)]@{--}"b2"
\POS?(0.5)="a",?(0.5)/-\intercol/="b","a"
*\xybox{"b"-"a":(0,0)\ar(0,1)\POS(0,1)*!LD{\ve b_3^*}}%
\POS"b2"\ar@[|(2)]@{--}"b3"
\POS?(0.5)="a",?(0.5)/-\intercol/="b","a"
*\xybox{"b"-"a":(0,0)\ar(0,1)\POS(0,1)*!RD{\ve b_1^*}}%
}
\end{xy}
=
\begin{xy}
<\intercol,0pt>:<0pt,\intercol>::
\POS(0,0)*[*0.85]\xybox{
\POS(2,0)="b1"*+!L{\ve b_1}
\POS(0,2)="b2"*+!D{\ve b_2}
\POS(-3,1)="w"*+!DR{\ve w}
\POS"b1"\ar@[|(2)]@{--}"b2"
\POS?(0.5)="a",?(0.5)/-\intercol/="b","a"
*\xybox{"b"-"a":(0,0)\ar(0,1)\POS(0,1)*!LD{\ve b_{0,3}^*}}%
\POS"b2"\ar@[|(2)]@{--}"w"
\POS?(0.5)="a",?(0.5)/-\intercol/="b","a"
*\xybox{"b"-"a":(0,0)\ar(0,1)\POS(0,1)*!RD{\ve b_{1,3}^*}}%
\POS"w"\ar@[|(2)]@{-}"b1"
\POS?(0.5)="a",?(0.5)/-\intercol/="b","a"
*\xybox{"b"-"a":(0,0)\ar(0,1)\POS(0,1)*!RU{\ve b_{2,3}^*}}%
}
\end{xy}
+
\begin{xy}
<\intercol,0pt>:<0pt,\intercol>::
\POS(0,0)*[*0.85]\xybox{
\POS(-2,-1)="b3"*+!U{\ve b_3}
\POS(2,0)="b1"*+!L{\ve b_1}
\POS(-3,1)="w"*+!DR{\ve w}
\POS"b3"\ar@[|(2)]@{-}"b1"
\POS?(0.5)="a",?(0.5)/-\intercol/="b","a"
*\xybox{"b"-"a":(0,0)\ar(0,1)\POS(0,1)*!RU{\ve b_{0,2}^*}}%
\POS"b1"\ar@[|(2)]@{--}"w"
\POS?(0.5)="a",?(0.5)/-\intercol/="b","a"
*\xybox{"b"-"a":(0,0)\ar(0,1)\POS(0,1)*!LD{\ve b_{3,2}^*}}%
\POS"w"\ar@[|(2)]@{--}"b3"
\POS?(0.5)="a",?(0.5)/-\intercol/="b","a"
*\xybox{"b"-"a":(0,0)\ar(0,1)\POS(0,1)*!RU{\ve b_{1,2}^*}}%
}
\end{xy}
-
\begin{xy}
<\intercol,0pt>:<0pt,\intercol>::
\POS(0,0)*[*0.85]\xybox{
\POS(-2,-1)="b3"*+!U{\ve b_3}
\POS(0,2)="b2"*+!D{\ve b_2}
\POS(-3,1)="w"*+!DR{\ve w}
\POS"b3"\ar@[|(2)]@{-}"b2"
\POS?(0.5)="a",?(0.5)/-\intercol/="b","a"
*\xybox{"b"-"a":(0,0)\ar(0,1)\POS(0,1)*!LU{\ve b_{0,1}^*}}%
\POS"b2"\ar@[|(2)]@{--}"w"
\POS?(0.5)="a",?(0.5)/-\intercol/="b","a"
*\xybox{"b"-"a":(0,0)\ar(0,1)\POS(0,1)*!RD{\ve b_{3,1}^*}}%
\POS"w"\ar@[|(2)]@{--}"b3"
\POS?(0.5)="a",?(0.5)/-\intercol/="b","a"
*\xybox{"b"-"a":(0,0)\ar(0,1)\POS(0,1)*!RU{\ve b_{2,1}^*}}%
}
\end{xy}
$$
\caption{Signed decomposition of a half-open $3$-dimensional simplicial cone.
Each cone is represented by a vertex figure.
Closed facets ($\sigma = 1$) are shown in solid lines,
while open facets ($\sigma = -1$) are shown in broken lines.}
\label{fig:decomposition}
\end{figure}

\end{example}

  Finally we remark that other constructions of~$\ve y$ are possible, giving rise
  to different combinatorial rules.  For instance, the implementation
  \texttt{barvinok}~\citep{barvinok-manual-noversion} uses a set of rules that
  correspond to a modification of~\eqref{eq:signed-decomp:lexico-y-revised},
  where for all $i\in J_\leq$ the coefficient $\gamma^i$ is replaced by $\gamma^{i+d}$.
  
\end{remark}

\section{Parametric Barvinok algorithm using exact decompositions in
the primal space}

In the previous section, we have shown how to both triangulate a closed
polyhedral cone (\autoref{s:subdivision}) and apply Barvinok's
decomposition (\autoref{s:decomposition}) in the primal space
without introducing (indicator functions of) lower-dimensional
polytopes.
The result is a signed sum of half-open simplicial cones.
The final remaining step in obtaining a generating function for a polytope
is therefore the computation of the generating function of such a cone.

\subsection{The generating function of a half-open simplicial rational cone}
\label{s:gf}

If $\ve v(\ve q) + C$ is a {\em closed\/} simplicial affine cone
where
$
C = \{\,
\sum_{j=1}^d \lambda_j \ve b_j : \lambda_j \ge 0
\,\}
$ with $\ve b_j\in\Z^d$,
then it is well known~\cite{stanley} that the generating function
$g_{\ve v(\ve q) + C}$ of $\ve v(\ve q) + C$
is
\begin{equation}
\label{eq:fundamental}
g_{\ve v(\ve q) + C}(\ve z)
= \frac{\sum_{\ve \alpha \in \Pi \cap \Z^d} \ve z^{\ve \alpha}}
		  {\prod_{j=1}^d \left(1 - \ve z^{\ve b_j}\right)}
,
\end{equation}
with the {\em fundamental parallelepiped\/} of $\ve v(\ve q) + C$,
\begin{displaymath}
\Pi = \ve v(\ve q) + \biggl\{\,
\sum_{j=1}^d \lambda_j \ve b_j : 0 \le \lambda_j < 1
\,\biggr\}
.
\end{displaymath}
For a half-open cone $\ve v(\ve q) + \tilde C$ given by~\eqref{eq:explicit},
the same formula holds with
$$
\Pi = \ve v(\ve q) + \biggl\{\,
\sum_{j=1}^d \lambda_j \ve b_j :
0 \le \lambda_j < 1 \text{  for $j\in J_\leq$ and }
0 < \lambda_j \le 1 \text{  for $j\in J_<$}
\,\biggr\}
.
$$

To enumerate all points in $\Pi \cap \Z^d$ and compute the numerator
of~\eqref{eq:fundamental},
we follow the technique of~\cite[Lemma 5.1]{barvinok-1993:exponential-sums}, 
which we adapt for the case of half-open cones.
\begin{lemma}
  Let $B$ be the matrix with the $\ve b_j$ as columns and let $S$ be the Smith
  normal form of $B$, i.e., $B V = W S$, with $V$ and $W$ unimodular matrices
  and $S$ a diagonal matrix $S = \diagonal \ve s$.  Then, 
  \begin{displaymath}
    \Pi \cap \Z^d =
    \{\, \ve \alpha(\ve k) : k_j \in \Z, 0 \le k_j < s_j \,\},
  \end{displaymath}
  with
  \begin{align*}
    \ve \alpha(\ve k) & = \ve v(\ve q) + \sum_{j \in J_\leq}
    \fractional{\inner{\ve b^*_j, \ve v(\ve q) - W \ve k}} \ve b_j + \sum_{j
      \in J_<} \cractional{\inner{\ve b^*_j, \ve v(\ve q) - W \ve k}} \ve b_j
    \\
    & = W \ve k - \sum_{j \in J_\leq} \floor{\inner{\ve b^*_j, \ve v(\ve q) -
        W \ve k}} \ve b_j - \sum_{j \in J_<} \ceil{\inner{\ve b^*_j, \ve v(\ve
        q) - W \ve k} - 1} \ve b_j ,
  \end{align*}
  with $\fractional{\cdot}$ the (lower) fractional part $\fractional{x} = x -
  \floor{x}$ and $\cractional{\cdot}$ the (upper) fractional part
  $\cractional{x} = x - \ceil{x-1} = 1 - \fractional{-x}$.
\end{lemma}
\begin{proof}
  It is clear that each $\ve \alpha(\ve k) \in \Pi \cap \Z^d$.  To see that
  all integer points in $\Pi$ are exhausted, note that $\det B = \det S$ and
  that all $\ve \alpha(\ve k)$ are distinct.  The latter follows from the fact
  that $\ve \alpha(\ve k)$ can be written as $\ve \alpha(\ve k) = W \ve k + B
  \ve \gamma = W \ve k + W S V^{-1} \ve \gamma$ for some $\ve \gamma \in
  \Z^d$.  If $\ve \alpha(\ve k_1) = \ve \alpha(\ve k_2)$, we must therefore
  have $\ve k_1 \equiv \ve k_2 \pmod{\ve s}$, i.e., $\ve k_1 = \ve k_2$.
\end{proof}

\subsection{Representations of the generating function of a parametric polytope}
\label{s:representations}

Let $Q_1,\dots,Q_k\subseteq Q$ be the chambers of the parameterized inequality
system~$A\ve x\leq\ve q$ of maximal dimension.  For all parameters $\ve q$ from any given
chamber~$Q_i$, the parametric polytope $P_{\ve q} = \{\,\ve x\in\R^d: A\ve x \leq \ve q\,\}$ has the same set of primal
feasible simplex bases.  Due to affine-linear dependencies in the set~$Q$ of
parameters, several primal feasible simplex bases can yield the same vertex of
the polytope~$P_{\ve q}$ on the whole chamber~$Q_i$.  By this mapping we obtain a
set~$V_i$ of parametric vertices $\ve v_j(\ve q)$ for $j\in V_i$ and
associated vertex cones $\ve v_j(\ve q) + C_j$.
Let us denote by $g_{\ve v_j(\ve q) + C_j}(\ve z)$ the parametric generating
function of the vertex cone at $\ve v_j(\ve q)$. 

 By Brion's
Theorem, we obtain the expression
\begin{equation}
  \label{eq:sum-over-active-vertices-1-open-chamber}
  g_{P_{\ve q}}(\ve z) 
  = \sum_{\ve v_j \in V_i} g_{\ve v_j(\ve q) + C_j}(\ve z)
\end{equation}
for the generating function of the parametric polytope
$P_{\ve q}$, valid for all parameters $\ve q\in Q_i$.
It turns out~\cite[Section~4.2]{Brion1997residue}
that the formula~\eqref{eq:sum-over-active-vertices-1-open-chamber} is also valid on the closure $\closure Q_i$ of
the chamber~$Q_i$.  In this way, we obtain the usual representation of the parametric
generating function as a \emph{piecewise function} defined on the whole parameter space~$Q$:
\begin{equation}
  \label{eq:piecewise}
g_{P_{\ve q}}(\ve z) =
\begin{cases}
\sum_{j \in V_{1}} g_{\ve v_j(\ve q) + C_j}(\ve z)
	& \text{if $\ve q \in \closure Q_1$}
\\
\quad\quad\quad\vdots
\\
\sum_{j \in V_{k}} g_{\ve v_j(\ve q) + C_j}(\ve z)
	& \text{if $\ve q \in \closure Q_k$}
.
\end{cases}
\end{equation}
As explained in \autoref{s:parametric:polytope}, this yields
possibly different expressions for values of $\ve q$ on the intersecting
boundaries of two or more chambers.\smallbreak

We are now interested in a different representation of the parametric generating
function, 
\begin{equation}
  \label{eq:piecewise-disjoint}
g_{P_{\ve q}}(\ve z) =
\begin{cases}
\sum_{j \in V_{1}} g_{\ve v_j(\ve q) + C_j}(\ve z)
	& \text{if $\ve q \in \tilde Q_1$}
\\
\quad\quad\quad\vdots
\\
\sum_{j \in V_{k}} g_{\ve v_j(\ve q) + C_j}(\ve z)
	& \text{if $\ve q \in \tilde Q_k$}
.
\end{cases}
\end{equation}
where the sets $\tilde Q_i$ form a \emph{partition} of the parameter space,
\begin{equation}
  \label{eq:Q-partition}
  Q = \tilde Q_1 \cup \dots \cup \tilde Q_k 
  \quad\text{with}\quad \tilde Q_i\cap \tilde Q_i' = \emptyset
  \text{ for $i\neq i'$}.
\end{equation}
The benefit of representation~\eqref{eq:piecewise-disjoint} is that it can be
rewritten in the form of a closed formula using indicator functions,
\begin{equation}
\label{eq:gf:chambers}
g_{P_{\ve q}}(\ve z) = \sum^k_{i=1} [\tilde Q_i](\ve q)
			\sum_{j \in V_{i}} g_{\ve v_j(\ve q) + C_j}(\ve z)
.
\end{equation}

Clearly, representations~\eqref{eq:piecewise-disjoint} and
\eqref{eq:gf:chambers} can be obtained by taking the chambers of \emph{all}
dimensions, since they form a partition of~$Q$.  However, we can do better:
\begin{lemma}
  We can construct representations \eqref{eq:piecewise-disjoint} and
  \eqref{eq:gf:chambers}, where $k$ is the number of chambers of~$A\ve x\leq
  \ve q$ of maximal dimension.  When the dimension~$d$ of the polytopes and
  the dimension~$m$ of the parameter space are fixed, the construction is
  possible in polynomial time.
\end{lemma}

\begin{proof}
  Again, we can apply the technique of \autoref{th:exactify-identities} to 
  define half-open polyhedra~$\tilde Q_i$ that
  satisfy~\eqref{eq:Q-partition}, where~$\ve y$ is now an arbitrary vector
  from the relative interior of one of the chambers of maximal dimension.
  The complexity in fixed dimensions $m$ and $d$ follows from the fact that
  there are only polynomially many full-dimensional chambers in this case.  
\end{proof}

Note that the generating function of a parametric vertex may
appear multiple times in representation~\eqref{eq:gf:chambers} since a
vertex~$\ve v_j(\ve q)$ 
may be active on more than one chamber.  
The multiple occurrences can be removed by considering
the \emph{activity regions} 
\begin{displaymath}
  A_j = \{\, \ve q : A \ve v_j(\ve q) \le \ve q \,\}
\end{displaymath}
of individual vertices instead of the chambers.
Then, by introducing their half-open counterparts $\tilde A_j$ constructed
by~\autoref{th:exactify-identities}, 
we obtain another representation of the parametric generating function,
\begin{equation}
\label{eq:gf:vertices}
g_{P_{\ve q}}(\ve z) = \sum_{j\in V} [\tilde A_j](\ve q) \, g_{\ve v_j(\ve q)
  + C_j}(\ve z).
\end{equation}
where $V=V_1\cup\dots\cup V_k$ is the index set of all appearing parametric
vertices. 
One advantage of this representation is that it can be computed
in polynomial time, even if the dimension $m$ of the parameter space varies:
\begin{lemma}
  The representation~\eqref{eq:gf:vertices} can be constructed in polynomial
  time when the dimension~$d$ of the polytopes is fixed (but the dimension~$m$
  of the parameter space varies).
\end{lemma}
\begin{proof}
  This follows from the above discussion; the number of parametric vertices is
  polynomial when the dimension~$d$ of the polytopes is fixed and the
  dimension~$m$ of the parameter space varies.
\end{proof}

\subsection{From the generating function to the counting function}
\label{s:specialization}

After computing the parametric generating function 
$g_{P_{\ve q}}(\ve z)$ of $P_{\ve q}$,
an explicit representation of the parametric
counting function $c(\ve q) = \# ( P_{\ve q} \cap \Z^d )$
can be obtained by evaluating the generating function at $\ve 1$,
i.e., $c(\ve q) = g_{P_{\ve q}}(\ve 1)$.
Care needs to be taken in this evaluation since $\ve 1$ is a pole
of each term in $g_{P_{\ve q}}(\ve z)$.
One typically computes the constant terms of the Laurent expansions
of these rational functions;
see~\cite{Barvinok94,BarviPom,latte1,verdoolaege-et-al:counting-parametric}. 

Applying this process to \eqref{eq:gf:chambers} and \eqref{eq:gf:vertices},
one obtains the counting formulas
$$
c(\ve q) = \sum_{i=1}^k [\tilde Q_i](\ve q) \, \sum_{\ve v_j \in V_i}
c_{\ve v_j(\ve q) + C_j} 
$$
and
$$
c(\ve q) = \sum_{\ve v_j} [\tilde A_j](\ve q) \, c_{\ve v_j(\ve q) + C_j}
,
$$
where $c_{\ve v_j(\ve q) + C_j}$ is the sum of the constant terms in the
Laurent expansions of the terms in $g_{\ve v_j(\ve q) + C_j}(\ve z)$.

\subsection{The resulting algorithms}
\label{s:algorithm-summary}
The complete resulting algorithm, based on a chamber decomposition, is shown
below.  

\begin{algorithm}[Primal parametric Barvinok algorithm]\mbox{}\\
{\bf Input:} full-dimensional parametric polytope
$P_{\ve q} = \{\,\ve x\in\R^d: A\ve x \leq \ve q\,\}$,
with $\ve q\in Q\subseteq\R^m$; the maximum enumerated cone index~$\ell$
\\
{\bf Output:} parametric counting function
  $c\colon Q \to \N$ with $c(\ve q) = \# \bigl( P_{\ve q} \cap \Z^d \bigr)$
\begin{enumerate}
\item Compute the chamber decomposition ${\mathcal Q} \subset 2^Q$
of $P_{\ve q}$
and for each $Q_i \in {\mathcal Q}$ of maximal dimension, the corresponding
active vertices $V_i = \{\, \ve v_j(\ve {q}) \,\}_j$
(see \autoref{s:parametric:polytope})
\item Compute half-open chambers~$\tilde Q_i$ from~$Q_i$.
\item For each vertex cone
$\ve v_j(\ve {q}) + C_j$
of $P_{\ve q}$, with
$\ve v_j(\ve {q}) \in \bigcup_{Q_i \in {\mathcal Q}} V_i$
\begin{enumerate}
\item Triangulate $C_j$ into half-open full-dimensional simplicial
cones $[C_j] = \sum_k [\tilde C_{jk}]$
(see \autoref{s:subdivision})
\item For each $\tilde C_{jk}$,
apply Barvinok's signed decomposition into half-open full-dimensional
cones $[\tilde C_{jk}] = \sum_l \epsilon_{jkl}[\tilde C_{jkl}]$ of index
at most $\ell$
(see \autoref{s:decomposition})
\item For each $\tilde C_{jkl}$, write down the generating
function $g_{\ve v_j(\ve {q}) + \tilde C_{jkl}}(\ve z)$~\eqref{eq:fundamental}
of the affine cone $\ve v_j(\ve {q}) + \tilde C_{jkl}$
(see \autoref{s:gf})
\item Write down
$g_{\ve v_j(\ve {q}) + C_j}(\ve z) =
\sum_k \sum_l \epsilon_{jkl} g_{\ve v_j(\ve {q}) + \tilde C_{jkl}}(\ve z)$
\end{enumerate}
\item Write down the generating function
$g_{P_{\ve q}}(\ve z)$~\eqref{eq:gf:chambers}
of the parametric polytope $P_{\ve q}$
(see \autoref{s:representations})
\item Specialize the generating function $g_{P_{\ve q}}(\ve z)$
to obtain the counting function $c(\ve q) = g_{P_{\ve q}}(\ve 1)$
(see \autoref{s:specialization})
\end{enumerate}
\end{algorithm}
We omit the variation based on activity regions, as it is nearly
identical.

%\clearpage

\bibliographystyle{plainnat}
\bibliography{barvinok,weismantel}

\end{document}